\documentclass[12pt, twoside, reqno]{amsart} 
 \usepackage{pxfonts}

\usepackage{latexsym}
\usepackage{amsmath}
\usepackage{amssymb}
\usepackage{amsthm}
\usepackage{amscd}
\usepackage{enumerate} 
\usepackage{amssymb} 
\usepackage{mathrsfs}     
\usepackage[all]{xy}
\usepackage{color}
\usepackage{graphicx}
\usepackage{hyperref}
\hypersetup{colorlinks=true}

\newcommand{\Spec}[0]{{\operatorname{Spec}}}

  \makeatletter
    
    \@addtoreset{equation}{section}
  \makeatother

\newtheorem{thm}{Theorem}[section]

\newtheorem{lem}[thm]{Lemma}

\newtheorem{rem}[thm]{Remark}

\newtheorem{cor}[thm]{Corollary}

\newtheorem{ques}[thm]{Question}

\theoremstyle{definition}

\newtheorem{defi}[thm]{Definition}

\newtheorem*{ack}{Acknowledgments} 

\newtheorem*{notation}{Notation} 

\theoremstyle{remark}

\newtheorem{step}{Step}

\title[Injectivity theorem for globally $F$-regular varieties]{Koll\'ar's injectivity theorem for globally $F$-regular varieties}

\author{Yoshinori  Gongyo}
\author{Shunsuke Takagi}
\address{Graduate School of Mathematical Sciences, the University of Tokyo, 3-8-1 Komaba, Meguro-ku, Tokyo 153-8914, Japan.}
\email{gongyo@ms.u-tokyo.ac.jp}
\email{stakagi@ms.u-tokyo.ac.jp}

\begin{document}
\bibliographystyle{amsalpha+}
 
 \maketitle
 
 \begin{abstract}
We prove Koll\'ar's injectivity theorem for globally $F$-regular varieties. 
\end{abstract}
 
\section{Introduction}
The following injectivity theorem for semi-ample line bundles was proved by Koll\'ar \cite[Theorem 2.2]{ko}: 
   
\begin{thm}[Koll\'ar's injectivity theorem]\label{kollar-inj-char0}
Let $X$ be an $n$-dimensional  smooth projective variety over an algebraically closed field of characteristic zero and $\mathcal{L}$ be a semi-ample line bundle on $X$, that is, $\mathcal{L}^{\otimes r}$ is base point free for some positive integer $r$. 
Let $s$ be a non-zero section in $H^0(X, \mathcal{L}^{\otimes \ell})$ for some positive integer $\ell$. 
Then the map 
\[
\times s: H^i(X, \omega_X \otimes \mathcal{L}^{\otimes m}) \to  H^i(X, \omega_X \otimes \mathcal{L}^{\otimes (\ell+m)})
\] 
induced by the tensor product with $s$ is injective for all $m\geq1$ and $i\geq0$, where $\omega_X$ is the canonical line bundle on $X$. 
\end{thm}
Koll\'ar's injectivity theorem can be viewed as a generalization of the Kodaira vanishing theorem. Indeed, combining Theorem \ref{kollar-inj-char0} with the Serre vanishing theorem, we can easily recover the Kodaira vanishing theorem. 
The reader is referred to \cite{fujino2} for a recent development of such an injectivity theorem in characteristic zero. 

This paper discusses what happens if the variety is defined over a field of positive characteristic.  
Koll\'ar's proof of Theorem \ref{kollar-inj-char0} depends on the Hodge decomposition and does not work in positive characteristic. 
Indeed, Raynaud \cite{raynaud} gave a counterexample to the Kodaira vanishing theorem in positive characteristic. 
On the other hand, Mehta-Ramanathan \cite{mehta--ramanathan} proved that the Kodaira vanishing theorem holds for (Cohen-Macaulay) globally $F$-split varieties, that is, for projective varieties in positive characteristic whose absolute Frobenius morphisms split globally.  
Moreover, Fujino \cite{fujino} proved that Koll\'ar's injectivity theorem holds for projective toric varieties in any characteristic. 
Since projective toric varieties are globally $F$-split in positive characteristic, it is natural to ask the following question: 
\begin{ques}\label{q1}
Does Koll\'ar's injectivity theorem hold for globally $F$-split varieties?
\end{ques} 

Inspired by the theory of $F$-singularities, Smith \cite{smith} introduced the notion of globally $F$-regular varieties (see Definition \ref{defi-gFreg} for the definition), a special class of globally $F$-split varieties. 
Projective toric and Schubert varieties in positive characteristic are examples of globally $F$-regular varieties (see \cite{smith}, \cite{LRPT} and \cite{Has}). 
A weak form of the Kawamata-Viehweg vanishing theorem \cite[Corollary 4.4]{smith}, which is stronger than the Kodaira vanishing theorem, holds for globally $F$-regular varieties.  
In this paper, as a partial positive answer to Question \ref{q1}, we prove that Koll\'ar's injectivity theorem holds for globally $F$-regular varieties: 

\begin{thm}[\textup{cf.~Theorem \ref{kollar-inj}}]\label{kollar-inj-intro}
Let $X$ be an $n$-dimensional globally $F$-regular projective variety over an $F$-finite field of characteristic $p>0$ and  $\mathcal{L}$ be a semi-ample line bundle on $X$. 
Let $s$ be a non-zero section in $H^0(X, \mathcal{L}^{\otimes \ell})$ for some positive integer $\ell$. Then the map 
\[
\times s: H^i(X, \omega_X \otimes \mathcal{L}^{\otimes m}) \to  H^i(X, \omega_X \otimes \mathcal{L}^{\otimes (\ell+m)})
\] 
induced by the tensor product with $s$ is injective for all $m\geq1$ and $i \geq 0$. 
Moreover, $H^j(X, \omega_X \otimes \mathcal{L}^{\otimes m})=0$ for all $m\geq1$ and  $j \not= n-\kappa$, where $\kappa= \kappa(X, \mathcal{L})$ is the Iitaka dimension of $\mathcal{L}$. 
\end{thm}

We can also prove a relative version of Theorem \ref{kollar-inj-intro} (see Theorem \ref{kollar-inj}). 
 The proof is inspired by Fujino's idea of using multiplication maps to prove various vanishing theorems on toric varieties (see \cite{fujino}), but we focus more attention on the Cohen-Macaulayness of the higher direct images of the adjoint bundle $\omega_X \otimes \mathcal{L}^{\otimes m}$ under the semi-ample fibration induced by $\mathcal{L}$. 

\begin{ack}
The authors were partially supported by JSPS KAKENHI $\#$26707002, 15H03611, 16H02141 and 17H02831.
They are grateful to Sho Ejiri, Nobuo Hara, Kenta Sato and Ken-ichi Yoshida for helpful comments. 
The authors are also indebted to the referee for thoughtful suggestions. 
The first author would like to thank the organizers of ``Workshop in Algebraic Geometry" held in Hanga Roa, Chile on December 18--22, 2016. 
\end{ack}

\begin{notation}
Throughout this paper, all rings are commutative rings with unity.  
A variety over a field $k$ means an integral separated scheme of finite type over $k$. 
We use without explanation standard notation and conventions of the book \cite{komo}. 
\end{notation}

 \section{Preliminaries and lemmas}\label{preliminaries}

In this section, we briefly review the definition and basic properties of globally $F$-regular varieties introduced by Smith \cite{smith}. 

Recall that  a field $k$ of characteristic $p>0$ is said to be \textit{$F$-finite} if the extension degree $[k:k^p]$ is finite. 

\begin{defi}[\textup{\cite{mehta--ramanathan}, \cite{smith}, \cite{HX}}]\label{defi-gFreg} 
Let $X$ be a normal variety over an $F$-finite field $k$ of characteristic $p>0$. 
\begin{enumerate}\renewcommand{\labelenumi}{$($\textup{\roman{enumi}}$)$}
\item $X$ is said to be \textit{globally $F$-split} if the Frobenius map $\mathcal{O}_X \to F^e_*\mathcal{O}_X$ splits as an $\mathcal{O}_X$-module homomorphism. 
\item $X$ is said to be \textit{globally $F$-regular} if for every effective Weil divisor $D$ on $X$, there exists an integer $e \ge 1$ such that the composite map 
\[
\mathcal{O}_X \to F^e_*\mathcal{O}_X \hookrightarrow F^e_*\mathcal{O}_X(D)
\]
of the $e$-times iterated Frobenius map $\mathcal{O}_X \to F^e_*\mathcal{O}_X$ and the natural inclusion $F^e_*\mathcal{O}_X \hookrightarrow F^e_*\mathcal{O}_X(D)$ 
splits as an $\mathcal{O}_X$-module homomorphism. 
\item 
Let $f: X \to T$ be a projective morphism to a variety $T$ over $k$. 
We say that $X$ is \textit{globally $F$-regular} over $T$ if there exists an affine covering $\{U_i\}_{i \in I}$ of $T$ such that $f^{-1}(U_i)$ is globally $F$-regular for all $i \in I$. 
\end{enumerate}
\end{defi}

\begin{rem}\label{global to local}
 Let $X$, $T$ and $k$ be as in Definition \ref{defi-gFreg}. 
 
$(1)$ It follows from an argument similar to the proof of \cite[Theorem 3.1]{HH} that Definition \ref{defi-gFreg} (iii) is independent of the choice of the affine covering $\{U_i\}_{i \in I}$. 
In particular, when $T$ is affine, $X$ is globally $F$-regular over $T$ if and only if $X$ is globally $F$-regular. 
 
$(2)$ Let $T \to S$ be a projective morphism of varieties over $k$.  
If $X$ is globally $F$-regular over $S$, then $X$ is globally $F$-regular over $T$. 
 
$(3)$ Globally $F$-regular varieties are Cohen-Macaulay (see for example \cite[Proposition 4.1]{smith}). 
\end{rem}

The following lemma is proved by an argument similar to \cite[Corollary 4.3]{smith}, but we include the proof here for the sake of completeness. 

\begin{lem}\label{vanishing}Let $f:X \to T$ be a projective morphism from a normal variety $X$ to a variety $T$ over an $F$-finite field $k$ of characteristic $p>0$. 
Suppose that  $\mathcal{L}$ is an $f$-nef line bundle on $X$.
If $X$ is globally $F$-regular over $T$, then $\mathbf{R}^if_*\mathcal{L}=0$ for all $i>0$.
\end{lem}

\begin{proof}
We may assume that $T$ is affine, say $T= \mathrm{Spec}\; R$, and $X$ is globally $F$-regular. 
Let $H$ be an $f$-ample effective Cartier divisor on $X$. 
First, we show that $H^i(X, \mathcal{L}^{\otimes m}(H))=0$  
for all $m \ge 0$ and $i>0$.
Indeed, $H^i(X,  (\mathcal{L}^{\otimes m}(H))^{\otimes p^e})=0$ for all $i>0$ and all sufficiently large $e$ by the Serre vanishing theorem.  
On the other hand,  
since $X$ is globally $F$-regular, 
we have an injective $R$-homomorphism
\[
H^i(X, \mathcal{L}^{\otimes m}(H)) \hookrightarrow H^i(X, F^{e}_*\mathcal{O}_X \otimes (\mathcal{L}^{\otimes m}(H))) \cong H^i(X,  (\mathcal{L}^{\otimes m}(H))^{\otimes p^e})
\] 
for all $i \ge 0$ and all $e \ge 1$. Thus, we obtain the desired vanishing. 

Now we will show that $H^i(X, \mathcal{L})=0$ for all $i>0$. 
There exists an integer $e \ge 1$ such that we have an injective $R$-homomorphism
\[
H^i(X, \mathcal{L}) \hookrightarrow H^i(X, F_*^{e}\mathcal{O}_X(H)\otimes \mathcal{L}) \cong H^i(X, \mathcal{L}^{\otimes p^e}(H))
\] 
for all $i \ge 0$, because  $X$ is  globally $F$-regular. 
Therefore, it follows from the above vanishing that $H^i(X, \mathcal{L})=0$ for all $i>0$. 
\end{proof}

As a corollary of Lemma \ref{vanishing}, we obtain a Grauert-Riemenschneider type vanishing theorem for globally $F$-regular varieties. 
\begin{cor}\label{gr type vanishing}
Let $f:X \to Y$ be  a surjective projective morphism of normal  varieties over an $F$-finite field of characteristic $p>0$ with $f_*\mathcal{O}_X=\mathcal{O}_Y$. 
Suppose that $X$ is globally $F$-regular over $Y$.  
Then $\mathbf{R}^if_*\mathcal{O}_X=0$ for all $i\geq1$. 
Moreover, it holds that $\mathbf{R}^{n-m}f_*\omega_X=\omega_Y$ and $\mathbf{R}^{i}f_*\omega_X=0$ for all $i \not = n-m$, where $\mathrm{dim}\,X=n $ and $\mathrm{dim}\,Y=m $.
\end{cor}

\begin{proof}
The former statement immediately follows from Lemma \ref{vanishing}. 
We will show the latter statement. 
We may assume that $Y$ is affine and $X$ is globally $F$-regular. 
Note that by Remark \ref{global to local} (3) and \cite[Proposition 1.2]{HWY}, $X$ and $Y$ are Cohen-Macaulay. 
Let $\omega_Y^\bullet$ and $\omega_X^\bullet:=f^{!}\omega_Y^{\bullet}$ be normalized dualizing complexes on $Y$ and $X$, respectively.  
Since $X$ is Cohen-Macaulay, 
\[
\mathbf{R}\mathscr{H}\kern -.5pt om_{\mathcal{O}_X} (\mathcal{O}_X, \omega_X^\bullet) \cong \omega_X^{\bullet} \cong \omega_X[n]. 
\]
Similarly, since $Y$ is Cohen-Macaulay and $\mathcal{O}_Y \cong \mathbf{R}f_*\mathcal{O}_X$, 
\[
\mathbf{R}\mathscr{H}\kern -.5pt om_{\mathcal{O}_Y}(\mathbf{R}f_*\mathcal{O}_X, \omega_Y^\bullet)  \cong \mathbf{R}\mathscr{H}\kern -.5pt om_{\mathcal{O}_Y}(\mathcal{O}_Y, \omega_Y[m]) \cong \omega_Y[m]. 
\]
Thus, by the Grothendieck duality, we have natural isomorphisms
\begin{align*}
\mathbf{R}f_* \omega_X[n] 
& \cong \mathbf{R}f_*\mathbf{R}\mathscr{H}\kern -.5pt om_{\mathcal{O}_X}(\mathcal{O}_X, \omega_X^\bullet)\\ 
& \cong \mathbf{R}\mathscr{H}\kern -.5pt om_{\mathcal{O}_Y}(\mathbf{R}f_*\mathcal{O}_X, \omega_Y^\bullet)\\ 
& \cong \omega_Y[m], 
\end{align*}
which implies that $\mathbf{R}^if_*\omega_X=0$ for all $i \not = n-m$ and $\mathbf{R}^{n-m}f_*\omega_X= \omega_Y$. 
\end{proof}

We also use the following elementary and purely algebraic lemma in the proof of the main theorem. 
\begin{lem}\label{CM lemma}
Let $(R, \mathfrak{m})$ be a Noetherian local domain with dualizing complex and $M$ be a finitely generated $R$-module. Then there exists a nonzero element $c \in R$ such that $c \cdot H^i_{\mathfrak{m}}(M)=0$ for all $i<\dim R$. 
\end{lem}
\begin{proof}
If $\dim M<\dim R$, then the annihilator $\mathrm{ann}_R(M)$ of $M$ is a nonzero ideal and its nonzero element is a desired element. 
If $\dim M=\dim R$, then the assertion follows from \cite[p.46]{sc}. 
\end{proof}

\section{Main Theorem}

In this section, we prove our main theorem, that is, Theorem \ref{kollar-inj}. 
Theorem \ref{kollar-inj-intro} is the special case of Theorem \ref{kollar-inj} where $S=\mathrm{Spec}\,k$. 

\begin{thm}\label{kollar-inj}Let $\pi:X\to S$ be a surjective projective morphism of varieties over an $F$-finite field $k$ of characteristic $p>0$, and suppose that $X$ is globally $F$-regular over $S$. 
Let $\mathcal{L}$ be a $\pi$-semi-ample line bundle on $X$ and $\kappa$ denote the relative Iitaka dimension of $\mathcal{L}$, that is, the Iitaka dimension $\kappa(X_{\eta}, \mathcal{L}_{\eta})$ of the generic fiber $X_{\eta}$ of $\pi$. 
Then the following holds:
\begin{enumerate}\renewcommand{\labelenumi}{$($\textup{\roman{enumi}}$)$}
\item $\mathbf{R}^i \pi_*(\omega_X \otimes \mathcal{L})=0$ for all $i \not= n-\kappa$, where $n=\dim X_{\eta}$. 
\item $\mathbf{R}^{n-\kappa} \pi_*(\omega_X \otimes \mathcal{L})$ is a torsion-free $\mathcal{O}_S$-module. 
\item Let $U$ be a non-empty Zariski open subset of $S$ and $s$ be a non-zero section in $H^0(U, \pi_*(\mathcal{L}^{\otimes \ell}))$ for some integer $\ell \ge 1$.
Then 
the map 
\[
\times s: \mathbf{R}^{n-\kappa}\pi_*(\omega_X \otimes \mathcal{L}^{\otimes m})|_{U} \to \mathbf{R}^{n-\kappa}\pi_*( \omega_X \otimes \mathcal{L}^{\otimes (\ell+m)})|_{U}
\] 
induced by the tensor product with $s$ is injective for every $m \geq 1$. 
\end{enumerate}
\end{thm}

\begin{proof} We may assume that $S=\mathrm{Spec}\, A$ is a $d$-dimensional affine variety and $X$ is a $(d+n)$-dimensional globally $F$-regular variety, because (i) and (ii) are local properties. 
Then it suffices to show that 
\begin{enumerate}\renewcommand{\labelenumi}{$($\textup{\roman{enumi}'}$)$}
\item $H^i(X, \omega_X \otimes \mathcal{L})=0$ for all $i \not= n-\kappa$, 
\item $H^{n-\kappa}(X, \omega_X \otimes \mathcal{L})$ is a torsion-free $A$-module, and 
\item 
if $s$ is a non-zero section in $H^0(X, \mathcal{L}^{\otimes \ell})$ for some integer $\ell \ge 1$, then 
the map 
\[
\times s: H^{n-\kappa}(X, \omega_X \otimes \mathcal{L}^{\otimes m}) \to  H^{n-\kappa}(X,  \omega_X \otimes \mathcal{L}^{\otimes (\ell+m)})
\] 
induced by the tensor product with $s$ is injective for every $m \geq 1$ as an $A$-module homomorphism. 
\end{enumerate}
Since $\mathcal{L}$ is semi-ample over $S$, there exist a surjective projective morphism $f: X \to  Y$  over $S$ with $f_*\mathcal{O}_X=\mathcal{O}_Y$ and a Cartier divisor $H$ on $Y$ such that $H$ is ample over $S$ and $f^*\mathcal{O}_Y(H) \cong \mathcal{L}^{\otimes r}$ for some integer $r \ge 1$. 
Note that $Y$ is a $(d+\kappa)$-dimensional globally $F$-regular variety by \cite[Proposition 1.2 (2)]{HWY}. 

Now we shall step into the detail of the proof,  we divide into four steps. 
\begin{step}\label{claim-cm}
We show that $f_*\mathcal{L}^{\otimes (-m)}$ is a maximal Cohen-Macaulay $\mathcal{O}_Y$-module for all integers $m \ge 1$.
\end{step}
For Step 1, we may assume that $Y$ is an affine variety, say $Y=\Spec \; R$.  
Fix an arbitrary closed point $y \in Y$ and we will show that the local cohomology $H^i_{\{y\}}(Y, f_*\mathcal{L}^{\otimes (-m)})=0$ vanishes for all $i < \dim Y=d+\kappa$. 
It follows from Lemma \ref{CM lemma} that there exists a nonzero element $c \in R$ such that $c \cdot H^i_{\{y\}}(Y, f_*\mathcal{L}^{\otimes (-b)})=0$ for all $i <d+\kappa$ and all $0 \le b \le r-1$. 
Put $D=\mathrm{div}_X(c)$. 
Since $X$ is globally $F$-regular, there exists an integer $e \ge 1$ such that the composite map 
\[
\mathcal{O}_X \to F^e_*\mathcal{O}_X \hookrightarrow F^e_*\mathcal{O}_X(D)
\]
splits as an $\mathcal{O}_X$-module homomorphism. 
Then tensoring with $\mathcal{L}^{\otimes (-m)}$ and taking the direct image by $f$ yield the splitting of the following composite map 
\[
f_*\mathcal{L}^{\otimes (-m)} \to F^e_*f_*\mathcal{L}^{\otimes (-mp^e)} \xrightarrow{\times F^e_*c} F^e_*f_*\mathcal{L}^{\otimes (-mp^e)} 
\]
as an $R$-module homomorphism. 
Furthermore, taking the $i$-th local cohomology, we find that the following composite map 
\begin{equation}\label{equation 1}
H^i_{\{y\}}(Y, f_*\mathcal{L}^{\otimes (-m)}) \to H^i_{\{y\}}(Y, f_*\mathcal{L}^{\otimes (-mp^e)}) \xrightarrow{\times c}  H^i_{\{y\}}(Y, f_*\mathcal{L}^{\otimes (-mp^e)}) \tag{$\star$}
\end{equation}
is injective. 
We write $mp^e=ar+b$ for integers $a$ and $b$ with $0 \le b \le r-1$. 
Since $f_*\mathcal{L}^{\otimes (-mp^e)} \cong f_*\mathcal{L}^{\otimes (-b)} \otimes \mathcal{O}_Y(-aH)$, we have an isomorphism $H^i_{\{y\}}(Y, f_*\mathcal{L}^{\otimes (-mp^e)}) \cong H^i_{\{y\}}(Y, f_*\mathcal{L}^{\otimes (-b)})$. 
Therefore, for each $i<d+\kappa$, 
\[
H^i_{\{y\}}(Y, f_*\mathcal{L}^{\otimes (-mp^e)}) \xrightarrow{\times c}  H^i_{\{y\}}(Y, f_*\mathcal{L}^{\otimes (-mp^e)})
\]
is the zero map by the choice of $c$, 
which implies by the injectivity of the map $(\star)$ that $H^i_{\{y\}}(Y, f_*\mathcal{L}^{\otimes (-m)})=0$ for all $i<d+\kappa$.

\begin{step}\label{claim1}
For all integers $m \ge 1$, we show the following two properties:  
\begin{enumerate}
\item $\mathbf{R}^{i}f_* (\omega_X \otimes \mathcal{L}^{\otimes m})=0$ for all $i \not= n-\kappa$.
\item $\mathbf{R}^{n-\kappa}f_* (\omega_X \otimes \mathcal{L}^{\otimes m})$ is a maximal Cohen-Macaulay $\mathcal{O}_Y$-module, so in particular a torsion-free $\mathcal{O}_Y$-module. 
\end{enumerate}
\end{step}

Let $\omega_Y^\bullet$ and $\omega_X^\bullet:=f^{!}\omega_Y^{\bullet}$ be normalized dualizing complexes on $Y$ and $X$, respectively.  
Since $X$ and $Y$ are Cohen-Macaulay by Remark \ref{global to local} (3), we have isomorphisms $\omega_X^{\bullet} \cong \omega_X[d+n]$ and $\omega_Y^{\bullet} \cong \omega_Y[d+\kappa]$. 
Note also that $\mathbf{R}f_*\mathcal{L}^{\otimes (-m)}\cong f_*\mathcal{L}^{\otimes (-m)}$ by Lemma \ref{vanishing}. 
Then, thanks to the Grothendieck duality, 
\begin{align*}
\mathbf{R}^{n-\kappa+j}f_*(\omega_X \otimes \mathcal{L}^{\otimes m}) 
& \cong H^{j-d-\kappa} (\mathbf{R} f_*(\omega_X^{\bullet} \otimes^{\mathbf{L}} \mathcal{L}^{\otimes m}))\\
& \cong H^{j-d-\kappa}(\mathbf{R}f_*\mathbf{R}\mathscr{H}\kern -.5pt om_{\mathcal{O}_X} (\mathcal{L}^{\otimes (-m)}, \omega_X^\bullet))\\
& \cong H^{j-d-\kappa} (\mathbf{R}\mathscr{H}\kern -.5pt om_{\mathcal{O}_Y} (\mathbf{R}f_*\mathcal{L}^{\otimes (-m)}, \omega_Y^\bullet))\\
& \cong H^{j}(\mathbf{R}\mathscr{H}\kern -.5pt om_{\mathcal{O}_Y} (f_*\mathcal{L}^{\otimes (-m)}, \omega_Y))\\
& \cong \mathscr{E}\kern -.5pt xt^{j}_{\mathcal{O}_Y} (f_*\mathcal{L}^{\otimes (-m)}, \omega_Y). 
\end{align*}
It follows from Step 1 and \cite[Theorem 3.3.10]{BH} that the Hom sheaf $\mathscr{H}\kern -.5pt om_{\mathcal{O}_Y}(f_*\mathcal{L}^{\otimes (-m)}, \omega_Y)$ is a maximal Cohen-Macaulay $\mathcal{O}_Y$-module and the Ext sheaf $\mathscr{E}\kern -.5pt xt^{j}_{\mathcal{O}_Y} (f_*\mathcal{L}^{\otimes (-m)}, \omega_Y)$ vanishes for all $j>0$. 
Therefore, $\mathbf{R}^{n-\kappa}f_*(\omega_X \otimes \mathcal{L}^{\otimes m})$ is a maximal Cohen-Macaulay module and $\mathbf{R}^{n-\kappa+j}f_*(\omega_X \otimes \mathcal{L}^{\otimes m})=0$ for all $j \ne 0$. 

\begin{step}
We now prove the assertions (ii') and (iii'). 
\end{step}

We consider the Leray spectral sequence: 
\[
E^{i,j}_2=H^i(Y, \mathbf{R}^jf_*(\omega_X \otimes  \mathcal{L}^{\otimes m})) \Rightarrow H^{i+j}(X, \omega_X \otimes \mathcal{L}^{\otimes m}). 
\]
By Step 2 (1), this spectral sequence induces an isomorphism 
\begin{equation}
H^{i}(Y, \mathbf{R}^{n-\kappa}f_*(\omega_X \otimes  \mathcal{L}^{\otimes m}) )\cong H^{i+n-\kappa}(X, \omega_X\otimes \mathcal{L}^{\otimes m})  \tag{$\star \star$}
\end{equation}
for all $i$. 
Then we obtain (ii') from Step 2 (2), because the direct image of a torsion-free sheaf by a surjective morphism is torsion-free. 

Take a nonzero element $t \in H^0(Y, \mathcal{O}_Y(\ell H))$ such that $f^*t=s^r$. 
The map induced by the tensor product with $t$
\[
\times t: H^{0}(Y, \mathbf{R}^{n-\kappa}f_*(\omega_X \otimes  \mathcal{L}^{\otimes m})) \to H^{0}(Y, \mathbf{R}^{n-\kappa}f_*(\omega_X \otimes \mathcal{L}^{\otimes m}) \otimes \mathcal{O}_Y(\ell H))
\]
is injective by Step 2 (2) again. 
Since 
\[\mathbf{R}^{n-\kappa}f_*(\omega_X \otimes \mathcal{L}^{\otimes m}) \otimes \mathcal{O}_Y(\ell H) \cong \mathbf{R}^{n-\kappa}f_*(\omega_X \otimes \mathcal{L}^{\otimes (\ell r+m)})\]
by the projection formula, 
the map $\times t$ can be identified via the isomorphism $(\star \star)$ with the map 
\[
\times s^r: H^{n-\kappa}(X, \omega_X \otimes \mathcal{L}^{\otimes m}) \to  H^{n-\kappa}(X,  \omega_X \otimes \mathcal{L}^{\otimes (\ell r+m)}).
\]
Thus, (iii') follows from the injectivity of the map $\times t$. 

\begin{step}
We finally show the assertion (i'). 
\end{step}

First, we verify that there exists an integer $e_0 \ge 1$ such that 
\[ H^{i+n-\kappa}(X, \omega_X\otimes \mathcal{L}^{\otimes p^{e_0}}) \cong H^{i}(Y, \mathbf{R}^{n-\kappa}f_*(\omega_X \otimes  \mathcal{L}^{\otimes p^{e_0}}) )=0\]
for all $i \ne 0$. 
The first isomorphism is nothing but $(\star \star)$. 
If $e \ge 1$ is an integer and we write $p^{e}=\alpha r+\beta$ for integers $\alpha$ and $\beta$ with $0\leq \beta \leq r-1$, then one has an isomorphism  
\[\mathbf{R}^{n-\kappa}f_*(\omega_X \otimes  \mathcal{L}^{\otimes p^e}) \cong \mathbf{R}^{n-\kappa}f_*(\omega_X \otimes  \mathcal{L}^{\otimes \beta})\otimes \mathcal{O}_X(\alpha H).\]
Thus, the existence of such an $e_0$ follows from the Serre vanishing theorem. 
On the other hand, by the global $F$-regularity of $X$, the map $\mathcal{O}_X \to F^e_*\mathcal{O}_X$ splits for every integer $e>0$. Taking its $\omega_X$-dual, we see that the trace map $\mathrm{Tr}_F^e: F^e_*\omega_X \to \omega_X$ of Frobenius is surjective and splits. 
The splitting of the trace map $\mathrm{Tr}_F^e$ induces the splitting of 
\[\mathrm{Tr}_F^e \otimes \mathcal{L}: F^e_*(\omega_X \otimes \mathcal{L}^{\otimes p^e}) \cong F^e_*\omega_X \otimes \mathcal{L} \to \omega_X \otimes \mathcal{L},\] 
which yields an injective map  
\[ H^{i+n-\kappa}(X, \omega_X \otimes \mathcal{L}) \to H^{i+n-\kappa}(X, \omega_X \otimes \mathcal{L}^{p^e }).
\]
for all $e>0$ and all $i$. 
Therefore, we obtain (i') from the vanishing of $H^{i+n-\kappa}(X, \omega_X \otimes \mathcal{L}^{\otimes p^{e_0}})$ for all $i \ne 0$. 
\end{proof}

We can reformulate Theorem \ref{kollar-inj} in the following way: 
\begin{cor}[\textup{cf.~\cite{EV}, \cite[Theorem 2.12]{fujino2}}]\label{cor-ev}
Let $\pi:X\to S$ be a projective morphism of varieties over an $F$-finite field of characteristic $p>0$ and $\mathcal{L}=\mathcal{O}_{X}(H)$ be a $\pi$-semi-ample line bundle on $X$. 
Suppose that $X$ is globally $F$-regular over $S$. 
Let $D$ be an effective Weil divisor on $X$ such that $tH \sim_{\mathbb{Q}, \pi}D+D'$ for some rational number $t>0$ and some effective Weil divisor $D'$ on $X$. 
Then the natural homomorphism 
\[\mathbf{R}^{i} \pi_*\mathcal{O}_X(\omega_X \otimes \mathcal{L})\to \mathbf{R}^{i} \pi_*\mathcal{O}_X(\omega_X \otimes \mathcal{L}(D))
\]
determined by $D$ is injective for all $i\geq0$.
\end{cor}

\begin{proof}Replacing $S$ by the image of $\pi$, we may assume that $\pi$ is surjective. 
Take positive integers $a$ and $b$ such that $aH\sim_{\mathbb{Z}, \pi} bD+bD'$.   
In order to prove the assertion of the corollary, it is enough to show that the map 
\begin{align*}
\mathbf{R}^{i} \pi_*\mathcal{O}_X(\omega_X \otimes \mathcal{L}) \to & \mathbf{R}^{i} \pi_*\mathcal{O}_X(\omega_X \otimes \mathcal{L}(aH))\\ 
\cong & \mathbf{R}^{i} \pi_*\mathcal{O}_X(\omega_X \otimes \mathcal{L}^{\otimes (a+1)})
\end{align*} 
is injective. 
However, once we take a section $s \in H^0(X, \mathcal{L}^{\otimes a})$ corresponding to $aH$, it immediately follows from Theorem \ref{kollar-inj}.  
\end{proof}

\begin{rem}
$(1)$ Since projective toric varieties are globally $F$-regular in positive characteristic (see for example \cite[Proposition 6.4]{smith}), Theorem \ref{kollar-inj} is a generalization of \cite[Theorem 4.1]{fujino} when the variety is projective over an $F$-finite field. 

$(2)$ A projective normal variety $X$ is said to be of Fano type if there exists an effective $\mathbb{Q}$-Weil divisor $\Delta$ on $X$ such that $(X, \Delta)$ is klt with $K_X+\Delta$ anti-ample. 
Schwede-Smith proved in \cite[Theorem 5.1]{SS} that if $X$ is a projective variety of Fano type over a field of characteristic zero, then its modulo $p$ reduction is globally $F$-regular for almost all $p$. 
Thus, Theorem \ref{kollar-inj} gives an alternative proof of Koll\'ar's injectivity theorem \cite[Theorem 2.2]{ko} for varieties of Fano type in characteristic zero.  
\end{rem}


\end{document}